\documentclass{amsart}

\usepackage{amsmath,amssymb,amsthm, amsfonts}
\usepackage[T1]{fontenc}
\usepackage[utf8x]{inputenc}
\usepackage{tikz}

\newcommand{\ie}{\textit{i.e.}\:}
\newcommand{\R}{\mathbb{R}}

\newcommand{\N}{\mathbb{N}}
\newcommand{\calP}{\mathcal{P}}

\newcommand{\Tau}{\mathrm{T}}
\newcommand{\indic}{\mathbf{1}}

\newcommand{\ve}{\varepsilon}
\newcommand{\impl}{\Rightarrow}
\newcommand{\deb}{\rightharpoonup}

\newcommand{\defeq}{\mathrel{\mathop:}=}
\newcommand{\eqdef}{=\mathrel{\mathop:}}
\newcommand{\abs}[1]{\left\lvert#1\right\rvert}
\newcommand{\norm}[1]{\left\lVert#1\right\rVert}
\newcommand{\set}[1]{\left\{#1\right\}}

\newcommand{\pical}{\mathcal{P}}

\DeclareMathOperator*{\Supp}{Supp}
\DeclareMathOperator*{\esssup}{ess\,sup}
\DeclareMathOperator*{\argmin}{arg\,min}

\DeclareMathOperator{\Lip}{Lip}

\numberwithin{equation}{section}

\theoremstyle{plain}
\newtheorem{thm}{Theorem}
\newtheorem{lem}[thm]{Lemma}
\newtheorem{prop}[thm]{Proposition}
\newtheorem{cor}[thm]{Corollary}
\newtheorem*{lem*}{Lemme}

\theoremstyle{definition}
\newtheorem{defn}{Definition}

\theoremstyle{remark}
\newtheorem*{rem}{Remark}

\tikzset { domaine/.style 2 args={domain=#1:#2} }
\tikzset{
xmin/.store in=\xmin, xmin/.default=-3, xmin=-3,
xmax/.store in=\xmax, xmax/.default=3, xmax=3,
ymin/.store in=\ymin, ymin/.default=-3, ymin=-3,
ymax/.store in=\ymax, ymax/.default=3, ymax=3,
}
\newcommand {\axes} {
\draw[->] (\xmin,0) -- (\xmax,0);
\draw[->] (0,\ymin) -- (0,\ymax);
}
\newcommand {\fenetre}
{\clip (\xmin,\ymin) rectangle (\xmax,\ymax);}

\begin{document}

\title[Optimal transport with oscillation costs]{Optimal transportation with an oscillation-type cost~: the one-dimensional case}
\author{Didier Lesesvre, Paul Pegon, Filippo Santambrogio}
\date\today

\begin{abstract}
The main result of this paper is the existence of an optimal transport map $T$ between two given measures $\mu$ and $\nu$, for a cost which considers the maximal oscillation of $T$ at scale $\delta$, given by $\omega_\delta(T):=\sup_{|x-y|<\delta}|T(x)-T(y)|$. The minimization of this criterion finds applications in the field of privacy-respectful data transmission. The existence proof unfortunately only works in dimension one and is based on some monotonicity considerations.  
\end{abstract}

\maketitle

\subjclass{{\bf MSC 2010} Primary: 49J45, Secondary: 49J05, 46N10}

\keywords{{\bf Keywords : }Monge-Kantorovich, Optimal Transportation, modulus of continuity, monotone transports, privacy respect}

\tableofcontents

\section{Introduction and motivations}

Optimal transport problems represent the mathematization of a very natural applied question, which is the following: given the initial density of a certain amount of mass, and the target density that we want to realize, which is the best possible way to displace the mass so as to guarantee a minimal cost? 

Based on an idea dating back to G. Monge (see \cite{Monge}), this is formalized through a map $T:X\to Y$ with the property $T_\#\mu=\nu$, where $\mu$ and $\nu$, probabilities on $X$ and $Y$, respectively, are the two given distribution of mass, represented by two measures (the case of densities is retrieved when the measures are absolutely continuous). The constraint  $T_\#\mu=\nu$, expressed in terms of the image measure (we recall that  $T_\#\mu$ is a measure defined through $T_\#\mu(A)=\mu(T^{-1}(A))$), stands for the fact that $T$ must ``rearrange'' the distribution $\mu$ into the new fixed one, $\nu$.

The typical criterion is based on the minimization of the average displacement
$$\min \int |T(x)-x|^p\,d\mu(x)$$
(we stick here to the euclidean case, where $X$ and $Y$ are subsets of $\R^d$ and $T(x)-x$ makes perfect sense, even if much has been said about other cases, in metric spaces, for instance). More generally, the criteria that have been studied are of the form $\int c(x,T(x))\,d\mu(x)$ for various cost-functions $c$, adapted to the different applications. For this wide class of problems, which is now known to be linked to many other branches of mathematics, from fluid mechanics, to mathematical economy, differential geometry, functional inequalities and probability, an alternative, convex, formulation is available thanks to the ideas of L. Kantorovich (see \cite{Kantorovich}). The existence of an optimal $T$ also passes first through this extended formulation.

These classical problems in optimal transport theory are now a very lively domain in pure and applied mathematics and most of them have already been solved or understood. Yet, it appears from some branches of applications, that a new generation of transport problems should be investigated, namely the minimization, over the same class of transport maps $T$, of more general costs, also depending on the differentiability or continuity properties of $T$. Let us think to optimization problems of the form 
$$\min \int |T(x)-x|^p\,d\mu(x)+\int |\nabla T(x)|^2dx,$$
where the goal is to find a good transport map, where its regularity also comes into play. Many variants of the form $\int L(x,T(x),\nabla T(x))dx$ could be considered, and they are already used in some applications, for instance in image processing (see \cite{PeyRab}) or shape analysis. Numerical studies on these issues have been performed, and not only in the last few years (see \cite{AngHakTanKik}, which anticipates a lot the current interest for this kind of problems in applied mathematics).  

The difficulties in studying this higher-order problem are somehow different than in the usual Kantorovich theory: here the existence of an optimal $T$ is typically easier to establish than in the usual theory, but much less is known about the characterization of the optimal maps. For instance, once we suppose that at least one map $T$ providing a finite value to the energy exists, it is not difficult to prove by standard compactness arguments in Sobolev spaces that a minimizer exists. Yet, finding it or studying its characterization appears to be more difficult, and \cite{LouSan} gives for instance a very partial answer, in the one dimensional and uniform case where $X\subset \R$ and $\mu$ is the Lebesgue measure over $X$.

Other criteria to minimize can easily appear to be meaningful, based for instance on the continuity of $T$ instead of its derivatives. If minimizing the Lipschitz constant $\Lip(T)$ is easy to understand (and a straightforward application of Ascoli-Arzel\`a Theorem provides existence of an optimal $T$), we want here to present a more tricky functional. For a fixed positive value $\delta$ one can consider the following quantity
 $$\omega_\delta (T)= \sup_{\abs{x-x'} < \delta} \abs{T(x) - T(x')};$$
which is merely the modulus of continuity of $T$ evaluated at $\delta$. Minimizing $\omega_\delta(T)$ means finding a transport map $T$ which is as continuous as possible, {\it at scale} $\delta$.

It is interesting to see that this very minimization problem 
$$\min\quad \{ \omega_\delta(T)\ | \ T_\#\mu=\nu \}$$
comes from a precise applied question which has been raised by the computer science community, in the framework of privacy-preserving protocols in telecommunications (see \cite{Mascetti2}).

Let us try to explain why this precise minimization should play a role in this setting, even if the reader could easily imagine other possible applications of this optimization problem.
Suppose that a high number of users are connected through their mobile phone to a service whose goal is to tell them whether their friends (from a social network list, for instance) are located or not within a certain fixed distance from them. To do so, all mobile phones periodically communicate their position to a common server which computes their distances and provides the users with the desired information. Yet, for privacy reasons, we do not want the server to know the position of each user, even if we want it to be able to compute their distances, which could seem difficult to realize. 
What is currently done to overcome this difficulty is that a third object, an external server, randomly chooses an isometry of the space (\ie, a rotation $R$ of the earth surface) and communicates it to the mobile phones of the users, but not to the server. The users then communicate their rotated position $R(x)$ instead of their position $x$ to the server, which is therefore able to compute their mutual distances, without exactly knowing their positions.

However, this is not satisfactory yet, since if the number of users is large enough, the distance-calculator server could see which are the densest regions\footnote{Notice that it is not always necessary that an external server chooses a same random isometry for all the users, since an alternative way of proceeding is that each pair of users (actually their mobile devices) secretly  agree on a randomly chosen isometry of the space that is not known to the server. In this case the server only knows mutual distances between users instead of their ``rotated'' positions, but this will be enough for him to make some clusters and compare to the distribution of the users in the world, which is well-known, see \cite{population grid}, and the final result would be the same.}
where most of the users are concentrated. It is not difficult to guess that these regions are more likely to correspond to Manhattan, Paris, Tokyo, and other strongly urbanized points of the Earth, and once one can reconstruct the position of these poles he can also reconstruct all the positions. This is a matter of non-uniform density, and this problem would not exist if the population on Earth was uniformly spread. Hence, an alternative idea could be the following: find a map $T$ (instead of $R$) transforming the given population density $\mu$ into a uniform density $\nu$; it is clear that $T$ cannot be an isometry, but one can look for the map $T$ which gives a minimal distance distortion. We cite for instance the work \cite{GraMad} where the distortion is minimized in terms - roughly speaking -  of the local bi-Lipschitz constant.

But the true model which is of interest for the privacy community is really the minimization of the $\omega_\delta$ modulus of continuity. Actually, if $\delta$ stands for the threshold distance the users are interested in, and we set $L=\omega_\delta(T)$, we get $|T(x)-T(y)|>L\impl |x-y|\geq \delta$. This means that all distances that are computed by the server to be larger than $L$ correspond for sure to original distances larger than $\delta$. The ambiguity stays true only in case the measured distance is smaller than $L$, but in such a case it is still possible to compute in a privacy-respectful way the distances, but with a more costly procedure (see \cite{Mascetti}, where cryptographic techniques are used to this aim). This finally means that it would be suitable to minimize $L$.

It is also clear that, from the application point of view, it is not really necessary to find the optimal map $T$, but any map with a small value for $\omega_\delta(T)$ would be fine. But exactly for this purposes it is worthwhile to study the minimization problem from a theoretical point of view, so as to find out possible general features of optimal maps (provided they exist) which could suggest how to produce ``good'' maps $T$.

This is why the present paper investigates, as a very first step in this research direction, the existence of an optimal map. This problem is quite hard since imposing continuity only at scale $\delta$ is not enough to give compactness for the minimizing sequences. 

On the positive side, as in the most classical traditions in transport problems, there exists a Kantorovich version of this problem, which is presented in this paper and existence of a minimizer for this extended problem is then proven. But the usual strategy consisting in proving that the Kantorovich minimizer actually derives from a map $T$ is not easy to implement.

A solution to this problem is proposed in dimension one, where it is possible to add an extra property: it is proven that an optimal map exists, and that it is piecewise monotone, where the number of monotonicity changes is at most of the order of $1/\delta$. 

In order to arrive to this result, the paper is organised as follows: Section 2 presents the key features of the problem we want to solve (the minimization among transport maps $T$) and of its Kantorovich relaxation, including the existence of a minimizer for this relaxed problem; Section 3 slowly gets to the existence of an optimal transport map starting from the optimal transport plan and applying suitable constructions; Section 4 gives an easy but interesting example where the optimal map is not monotone, but only piecewise monotone; finally, we describe in Section 5 how to handle some mathematical extensions of the problem that we preferred not to introduce from the beginning so as to make the paper more readable and to concentrate on the main ideas without too many technicalities.

\section{The original problem and its Kantorovich formulation}

Let $\delta$ be a positive real number and $\mu,\nu$ two probability measures on $\R^d$. Let us denote $\Omega$ and $\Omega'$ the supports $\Supp(\mu)$ and  $\Supp(\nu)$ of the two measures, and suppose for simplicity that they are compact (see Section 5 for the adaptations to a non-compact setting). All the functions that we consider will be defined on $\Omega$ (this will also be discussed in Section 5). We are interested in minimizing the functional
\begin{equation*}\tag{$M_\delta$}
\omega_\delta : T \longmapsto \sup_{\begin{array}{c}\abs{x-x'} < \delta\\ x,x'\in\Supp(\mu)\end{array}} \abs{T(x) - T(x')}
\end{equation*}
where $T$ lies in $\Tau(\mu,\nu) \defeq \{T \in \mathcal{B}(\Omega;\Omega') : T_\# \mu = \nu \}$, the set of \textbf{transport maps} from $\mu$ to $\nu$. The quantity $\omega_{\delta}$ is merely the modulus of continuity of $T$ evaluated at $\delta$.

\subsection{Kantorovich formulation}

Inspired from the Kantorovich reformulation of the initial Monge problem, we can consider the minimization on a wider class of objects : probability measures on the product space, or \textbf{transport plans}, instead of transport maps.  This should guarantee easier existence results and leads us to study the relaxed functional 
\begin{equation*}\tag{$K_\delta$}
\omega^K_{\delta} : \gamma \longmapsto (\gamma\otimes\gamma)-\esssup \left\{\abs{y-y'} : (x,y), (x',y') \in \Omega\times\Omega', \abs{x-x'}<\delta\right\}
\end{equation*}
where $\gamma$ lies in $\Pi(\mu,\nu) \defeq \left\{ \gamma \in \mathcal{P}(\Omega \times\Omega') : (\pi_1)_\# \gamma = \mu, \ (\pi_2)_\# \gamma = \nu \right\}$. This map can be expressed as 
$$\omega_{\delta}(\gamma) = \norm{f}_{L^{\infty}(\gamma \otimes \gamma)},$$
where
$$ f : (a,b) \in (\Omega\times\Omega') \times (\Omega\times\Omega') \longmapsto |\pi_2(b)-\pi_2(a)|\indic_{B(0,\delta)}(\pi_1(b) - \pi_1(a)).$$
Since the above function $f$ is lower semicontinuous, looking at its supremum on any set gives the same result as the supremum on the closure of the same set; in particular, its essential supremum coincides with the supremum on the support of the measure, hence we have
$$\omega^K_{\delta}(\gamma) = \sup \left\{\abs{y-y'} : (x,y), (x',y') \in \Supp(\gamma), \abs{x-x'}<\delta\right\}.$$

In the following, all the topological notions on $\mathcal{P}(\Omega)$, where $\Omega$ is compact, will relate to the weak-$\star$ topology with the identification to the dual of functions vanishing at infinity, \ie $\mathcal{C}_0(\Omega)'$ which equals $\mathcal{C}(\Omega)'$ for $\Omega$ is compact. 

%The proofs might be adapted to the non-compact case using the narrow topology induced by $\mathcal{C}_b(X)$ and the tightness of $\Pi(\mu,\nu)$.

The entire paper is devoted to the existence proof of an optimal transport map when $\mu$ has no atom and $\Omega\subset\R$. Further effort should still be made to handle the multi-dimensional case.

\subsection{Existence of an optimal transport plan}

\begin{thm}\label{thm:existplan}
Given two probabilities $\mu,\nu$ on $\Omega$ and $\Omega'$ respectively, and $\delta > 0$, there exists a transport plan  $\gamma \in \Pi(\mu,\nu)$ minimizing the cost $\omega^K_\delta$.
\end{thm}

\begin{proof}%{\bf Alternative proof}

Let us take a minimizing sequence $\gamma_n \in \Pi(\mu,\nu)$ and set $\Gamma_n=\Supp(\gamma_n)$. Up to subsequences, we can suppose weak-$\star$ convergence $\gamma_n\deb\gamma$ and Hausdorff convergence $\Gamma_n\to \Gamma$. It is clear that we have $\Supp(\gamma)\subset \Gamma$ and hence
$$\omega^K_\delta(\gamma)\leq \sup \left\{\abs{y-y'} : (x,y), (x',y') \in \Gamma, \abs{x-x'}<\delta\right\}.$$
Now, take two arbitrary points $(x,y),(x',y')\in\Gamma$ with $|x-x'|<\delta$. By Hausdorff convergence, it is possible to build two sequences $(x_n,y_n),(x'_n,y'_n)\in\Gamma_n$ with $(x_n,y_n)\to (x,y)$ and $(x'_n,y'_n)\to (x',y')$. In particular, for $n$ large enough, they satisfy $|x_n-x'_n|<\delta$ and hence $|y_n-y'_n|\leq \omega^K_\delta(\gamma_n)$, which yields $|y-y'|\leq \liminf_n \omega^K_\delta(\gamma_n)$. Passing to the supremum over the points of $\Gamma$ we get $\omega^K_\delta(\gamma)\leq \liminf_n \omega^K_\delta(\gamma_n)$. This semicontinuity proves that $\gamma$ is optimal.
\end{proof}

\section{Existence of an optimal transport map}

In the following, we denote by $K$ the minimal value of $ \omega^K_\delta$ on $\Pi(\mu,\nu)$, by $\gamma^\star$ an arbitrary minimizer for this relaxed problem and by $\Gamma$ its support : 
%\begin{align*}
$$K \defeq \min_{\Pi(\mu,\nu)} \omega^K_\delta,\quad
\gamma^\star \in \argmin_{\Pi(\mu,\nu)} \omega^K_\delta,\quad
\Gamma \defeq \Supp\gamma^\star.\\
%\Gamma_1 &\defeq \pi_1(\Gamma).
%\end{align*}
$$
Notice that the projection $\pi_1(\Gamma)$ of the support on the first factor of $\Omega\times\Omega'$ equals $\Supp(\mu)$, \ie $\Omega$.

We shall now state the main theorem of this paper:

\begin{thm} \label{thm:existmap}
Let $\Omega,\Omega'$ be compact set in $\R$, $\mu \in \calP(\Omega)$ and $\nu \in \calP(\Omega')$ two given probabilites on them. If $\mu$ has no atom, there exists an optimal transport map $T \in \Tau(\mu,\nu)$ for the cost $\omega_{\delta}$, where
\begin{equation*}
\omega_\delta : T \longmapsto \sup_{\begin{array}{c}\abs{x-x'}<\delta\\x,x'\in\Omega\end{array}} \abs{T(x)-T(x')}.
\end{equation*}
\end{thm}

\subsection{Preliminary remarks and definitions}

We denote by $f$ and $g$ the two functions defined on $\pi_1(\Gamma)=\Omega$ delimiting the convex hull along $y$ of $\Gamma$ \ie
\begin{align*}
 f(x) & \defeq \inf \set{y : (x,y) \in \Gamma},\\ 
 g(x) & \defeq \sup \set{y : (x,y) \in \Gamma}.
\end{align*}

\begin{rem}
These functions are Borel measurable and in particular $f$ is lower semicontinuous and $g$ is upper semicontinuous. Actually, since $\Gamma$ is compact, for every $x$ there is $y$ such that $(x,y)\in\Gamma$ and $f(x)=y$. If one takes $x_n\to x$ and $y_n=f(x_n)$ then there is a subsequence $y_{n_k}\to y=\liminf_n y_n$.  Since $\Gamma$ is closed, we get $(x,y)\in\Gamma$. Then one has $f(x)\leq y$ and semicontinuity is proven. Upper semicontinuity for $g$ is completely analogous. 
\end{rem}

The strip  $[f, g] \defeq \set{(x,y) \in \Omega \times \R : f(x) \leq y \leq g(x)}$ satisfies the following properties~:
\begin{gather*}
 \Gamma \subset [f,g]\tag{Inc}\label{eq:inc}\\
 \forall (x,y),(x',y') \in [f,g], \quad (\abs{x-x'}<\delta \Longrightarrow \abs{y-y'}\leq K)\tag{Opt}\label{eq:opt}
\end{gather*}

\begin{defn}A strip $[\phi, \psi]$ where $\phi, \psi :  \Omega \to \R$ is said \textbf{optimal} if it satisfies the optimality property \eqref{eq:opt}, and \textbf{admissible} if it satisfies both the inclusion \eqref{eq:inc} and optimality \eqref{eq:opt} properties.
\end{defn}

The functions $f$ and $g$ have not \textit{a priori} extra regularity properties than simply being lower and upper semicontinuous. The aim is to replace them with more regular ones, defining another admissible strip. Once such a nicer strip is constructed, we will prove that it contains the graph of a transport map, thus completing the proof, because the optimality property guarantees that any transport plan living in the strip has minimal cost.

Since $[f,g]$ is optimal, we have $\abs{f(x)-g(x')} \leq K$ whenever $\abs{x-x'}<\delta$, yielding~
\begin{subequations}\label{eq:fgineq}
 \begin{align}
 f(x) & \geq \inf_{\abs{x-x'}<\delta} g(x') - K,\\
 g(x) & \leq \sup_{\abs{x-x'}<\delta} f(x') + K,
\end{align}
\end{subequations}
motivating the next definition.

\begin{defn}\label{defn:transform}
For every function $\phi : \Omega\to \R$, we define its $\uparrow$ and $\downarrow$ transforms\footnote{These transforms also depend on the domain $\Omega$ which is used to define the $\sup$ and the $\inf$, but we will omit this dependence, which will be implicit throughout the paper, thus avoiding writing as $\phi^{\uparrow,\Omega}$ and similar heavy notations.} as follows~:
\begin{subequations}
\begin{align}
  \phi^\downarrow(x) &= \sup_{y\in \Omega, \abs{y-x}<\delta}  \phi(y) \label{eq:transf1},\\
  \phi^\uparrow(x) &= \inf_{y\in \Omega, \abs{y-x}<\delta} \phi(y). \label{eq:transf2}
\end{align}
\end{subequations}
\end{defn}

By definition of the transforms, the strips $[f,f^\uparrow +K]$ and $[g^\downarrow -K,g]$ are optimal. Moreover, rewriting \eqref{eq:fgineq}, one gets $g^\downarrow - K \leq f \leq g \leq f^\uparrow + K$, meaning that $f^\uparrow$ and $g^\downarrow$ are respectively the largest and smallest functions such that the strips $[f,f^\uparrow +K]$ and $[g^\downarrow -K, g]$ are optimal. In particular, they also keep satisfying the inclusion property. This is stated in the next proposition.

\begin{prop}
For all optimal strip $[\phi, \psi]$ where $\phi, \psi : \Omega\to \R$, $\phi^\uparrow$ and $\phi^\downarrow$ are respectively the largest and smallest functions such that $[\phi,\phi^\uparrow +K]$ and $[\psi^\downarrow -K, \psi]$ are optimal. Therefore, if $[\phi, \psi]$ is admissible, then these strips are also admissible.
\end{prop}

So far, we have shown that the admissible strip $[f,g]$ can be replaced by $[f, f^\uparrow +K]$, which is still admissible. Then it can be replaced by $[f^{\uparrow\downarrow}, f^\uparrow +K]$, since $(f^\uparrow + K)^\downarrow -K = f^{\uparrow\downarrow}$, thus enlarging the strip twice. One could wonder if this construction should go on, and the answer is negative, as a consequence of the following proposition.

\begin{prop}\label{prop:invconj}
For all $\phi : \Omega \to \R$,
\begin{align*}
\phi^{\uparrow\downarrow\uparrow} &= \phi^\uparrow,\\
\phi^{\downarrow\uparrow\downarrow} &= \phi^\downarrow.
\end{align*}
\end{prop}

\begin{proof}
Since $\phi^{\uparrow\downarrow}$ is the lowest possible, $\phi^{\uparrow\downarrow} \leq \phi$, hence $\phi^{\uparrow\downarrow\uparrow} \leq \phi^\uparrow$ by monotonicity of the transforms.
As for the converse inequality, $\forall y \in B(x,\delta)$,
$$\phi^{\uparrow\downarrow}(y) \doteq \sup_{z \in B(y,\delta)} \phi^\uparrow(z) \geq \phi^\uparrow(x)\\$$
so that by taking the infimum over $B(x,\delta)$,
$$\phi^{\uparrow\downarrow\uparrow}(x) \geq \phi^\uparrow(x).\qedhere$$
\end{proof}

This shows that the strip $[f^{\uparrow\downarrow},f^\uparrow + K]$, which is admissible, cannot be enlarged without losing the optimality condition. We will see that the $\uparrow$ and $\downarrow$ transforms have regularizing properties which justify the replacement of $[f,g]$ by $[f^{\uparrow\downarrow}, f^\uparrow + K]$.

\subsection{The proof}

We shall now study further properties of the transform operations defined previously and prove a few lemmas which will be useful in the final proof at the end of this section.

\begin{defn}
Let $\phi,\psi : \Omega\to \R$. We say that $(\phi,\psi)$ is a conjugate pair if $\phi^\uparrow = \psi$ et $\psi^\downarrow = \phi$.
\end{defn}

\begin{rem}
By Proposition \ref{prop:invconj}, $(\phi^{\uparrow\downarrow},\phi^\uparrow)$ is a conjugate pair.
\end{rem}

\begin{prop}\label{prop:transfregul}
Let $\phi : \Omega\to \R$ a Borel function. Then, $\phi^\uparrow$ and $\phi^\downarrow$ are regulated functions\footnote{A linear combination of characteristic functions of measurable sets will be called a simple function, and if these sets are intervals we call it a step function. A regulated function is by definition a uniform limit of step functions. This notion coincides, by the way, with that of functions having left and right-sided limits at every point. Note that we work here with functions defined on $\Omega$, where intervals are traces on $\Omega$ of intervals of $\R$. Therefore, regulated functions on $\Omega$ are exactly restrictions on $\Omega$ of regulated functions on $\R$.}, hence continuous outside a countable set. Moreover they are respectively lower and upper semicontinuous.
\end{prop}
\begin{proof}
Since $\phi$ is measurable and bounded, it can be expressed as the uniform limit of simple functions $\phi_n$. Now let us see how the $\uparrow$ and $\downarrow$ transforms act on simple functions. Let $\phi$ a simple function
$$\phi = \sum_{i=1}^N \alpha_i \indic_{A_i} \quad \text{where} \quad \alpha_1 < \cdots < \alpha_N.$$
We shall calculate its $\uparrow$ transform $\phi^\uparrow$. Since its value at some point $x$ is defined by the infimum of $\phi$ on $B(x,\delta)$, $\phi^\uparrow$ is equal to $\alpha_1$ at any point which is distant from $A_1$ by less than $\delta$, \ie on the set $\Omega \cap A_1^\uparrow$ where $A_1^\uparrow \defeq A_1 + B(0,\delta)$. It is a disjoint reunion of intervals of length at least $2\delta$, hence the reunion is finite. A similar reasoning tells us that $(\phi^\uparrow)^{-1}(\alpha_2)$ is equal to the trace of $A_2^\uparrow \defeq (A_2 + B(0,\delta))\setminus A_1^\uparrow$ on $\Omega$, and more generally that $\phi^\uparrow$ equals $\alpha_k$ on $\Omega \cap A_k^\uparrow$ where
$$A_k^\uparrow \defeq (A_k + B(0,\delta)) \setminus \left(\bigcup_{i=1}^{k-1} A_i^\uparrow \right)$$
which is a finite and disjoint reunion of intervals. Furthermore, since $\phi$ is supposed to be the limit of $(\phi_n)_n$ and by monotonicity of the $\uparrow$ transform, if $\norm{\phi_n - \phi}_\infty \leq \varepsilon$ then
\begin{equation*}
 \phi - \varepsilon \leq \phi_n \leq \phi +\ve \quad \text{hence} \quad (\phi - \varepsilon)^\uparrow = \phi^\uparrow - \varepsilon \leq \phi_n^\uparrow \leq (\phi+\ve)^\uparrow= \phi^\uparrow + \varepsilon.
\end{equation*}
It proves that $\phi^\uparrow$ is the uniform limit of $\phi_n^\uparrow$, which are step functions~: it is a regulated function. It is a classical issue to see that the set of discontinuity points of any regulated function is at most countable. The same result naturally holds for the $\downarrow$ transform of $\phi$.

Now we shall prove that $\phi^\uparrow$ is upper semicontinuous. Consider a sequence $(x_n)_n$ converging to $x$. Since $\phi^\uparrow = \inf_{B(\,\cdot\,,\delta)} \phi$, if $a\in B(x,\delta)$ then for $n$ large enough $|a-x_n| < \delta$, \ie $a \in B(x_n,\delta)$, and $\phi^\uparrow(x_n) \leq \phi(a)$. This implies that $\varlimsup \phi^\uparrow(x_n) \leq \phi(a)$ for all $a \in B(x,\delta)$, yielding
$$\varlimsup \phi^\uparrow(x_n) \leq \inf_{a\in B(x,\delta)} \phi(a) \doteq \phi^\uparrow(x)$$
which proves the upper semicontinuity of $\phi^\uparrow$. Since $\phi^\downarrow = -(-\phi)^\uparrow$, we get lower semicontinuity for $\phi^\downarrow$.
\end{proof}

In the following, all intervals, possibly given in the form $|a,b|$  where $|$ is either $[$ or $]$, will be intervals of $\Omega$.

\begin{lem}\label{lem:step}
Let $\phi,\psi : \Omega\to \R$ conjugate step functions, $\psi=\phi^\uparrow$ being expressed as
$$\sum_{j=1,\ldots, N} \alpha_j \indic_{I_j}$$
where the $I_j$'s are intervals partitioning $\Omega$ and sorted in increasing order, and the $\alpha_j$ are consecutively distinct. If $(\alpha_{k-1}, \alpha_k, \alpha_{k+1})$ is a triple such that $\alpha_{k-1} > \alpha_k$, $\alpha_{k+1} > \alpha_k$ (we say that $I_k$ is a \textbf{floor}), then the distance between $I_{k+1}$ and $I_{k-1}$, defined as $\inf I_{k+1}-\sup I_{k-1}$ (which could be larger than the measure of $I_k$, or even of its diameter, if $\Omega$ is disconnected) is at least $2\delta$. The symmetric result on \textbf{ceilings} of $\phi$ holds.
\end{lem}

\begin{proof}
By contradiction, assume that  the distance between $I_{k+1}$ and $I_{k-1}$ is smaller than $2\delta$. Pick a point $x$ in $I_k$ and $a \in \Omega$ such that $d(a,x) < \delta$. The ball $B(a,\delta)$  contains the point $x\in I_k$ but is wider than the distance between the two intervals $I_{k+1}$ and $I_{k-1}$: hence it intersects either $I_{k-1}$ or $I_{k+1}$ at a certain point $b$. Knowing that $(\phi,\psi)$ is conjugate, it follows that $\phi(a)$, which equals $\psi^\downarrow(a) \doteq \sup_{B(a,\delta)} \psi$, is greater or equal than $\psi(b)$, hence $\phi(a) \geq \min (\alpha_{k-1}, \alpha_{k+1})$. This is true for all $a$ in $B(x,\delta)$, so that $\phi \geq \min (\alpha_{k-1},\alpha_{k+1})$ on $\Omega\cap B(x,\delta)$, and $\psi(x) \geq \min (\alpha_{k-1},\alpha_{k+1}) > \alpha_k$ for $\psi(x) = \phi^\uparrow(x) \doteq \inf_{B(x,\delta)} \phi$. This cannot be true.
\end{proof}
%
%\begin{lem*}[Ancienne version]
%Let $\phi,\psi : \Omega\to \Omega$ conjugate step functions, $\psi$ being expressed as
%$$\sum_{j=1,\ldots, N} \alpha_j \indic_{I_j}$$
%where the $I_j$'s are intervals partitioning $\Omega$ and sorted in increasing order, and the $\alpha_j$'s are consecutively distinct. If $(\alpha_{k-1}, \alpha_k, \alpha_{k+1})$ is a triple such that $\alpha_{k-1} > \alpha_k$, $\alpha_{k+1} > \alpha_k$ (we say that $I_k$ is a \textbf{floor}), then for every $x \in I_k$, there exists some ball $B(y,\delta)$ containing $x$ where $\psi$ equals $\alpha_k$. In particular the neighbouring steps on the left and on the right of each floor are distant by at least $2\delta$. The symmetric result on \textbf{ceilings} of $\phi$ holds.
%\end{lem*}
%
%\begin{proof}
%By contradiction, assume that there is a point $x$ such that there is no such ball. Take $a \in \Omega$ such that $d(a,x) < \delta$. $B(a,\delta)$ is a ball of radius $\delta$ containing $x$, hence $\psi \neq \alpha_k$ on $B(a,\delta)$ and $B(a,\delta)$ intersects either $I_{k-1}$ or $I_{k+1}$ at a certain point $b$. Knowing that $(\phi,\psi)$ is conjugate, it follows that $\phi(a)$, which equals $\psi^\downarrow(a) \doteq \sup_{B(a,\delta)} \psi$, is greater or equal than $\psi(b)$, hence $\phi(a) \geq \min (\alpha_{k-1}, \alpha_{k+1})$. This is true for all $a$, so that $\phi \geq \min (\alpha_{k-1},\alpha_{k+1})$ on $B(x,\delta)$, and $\psi(x) \geq \min (\alpha_{k-1},\alpha_{k+1}) > \alpha_k$ for $\psi(x) = \phi^\uparrow(x) \doteq \inf_{B(x,\delta)} \phi$. This cannot be true.
%\end{proof}

This result will allow us to control the amount of monotonicity changes of $f$ and $g$ \textit{uniformly}, \ie in terms of $\Omega$ and $\delta$ and independently of the functions $f$ and $g$.

\begin{cor}\label{cor:monostep}
The number of floors $F_1 \leq \ldots \leq F_N$ of $\psi$ is bounded by a constant $M(\Omega,\delta)$, independent from $\phi,\psi$. Moreover, $\psi$ is nondecreasing then decreasing on each interval between floors, namely it increases on $G_1, \ldots, G_{N+1}$ and decreases on $H_1, \ldots, H_{N+1}$ where
$$G_1 \leq H_1 \leq F_1 \leq G_2 \leq H_2 \leq \ldots \leq F_N \leq G_{N+1} \leq H_{N+1}$$
is a subdivision of $\Omega$ in $3N+2$ intervals. A symmetric statement holds for $\phi$.
\end{cor}

\begin{proof}
We use the same notations as in Lemma \ref{lem:step}~:
\begin{equation*}
 \psi = \sum_{j=1,\ldots, N} \alpha_j \indic_{I_j}
\end{equation*}
where the $I_j$'s are sorted in increasing order and the $\alpha_j$'s are consecutively distinct. We have shown that floors separate their adjacent steps by a distance of at least $2\delta$. Since $\Omega$ is bounded, say it has length $L$, they cannot be more than
$$\left\lfloor \frac{L}{2\delta} \right\rfloor \eqdef M(\Omega,\delta).$$
Now, the union of the intervals between floors (or reaching an endpoint of $\Omega$) can be cut into two parts, $\psi$ being nondecreasing on the left one, then decreasing on the right one, because floors separate higher neighbouring steps by definition. The number of intervals in this subdivision is bounded from above by
$$3\left\lfloor \frac{L}{2\delta} \right\rfloor +2 \eqdef M'(\Omega, \delta).\qedhere$$
\end{proof}

Now, we shall extend this result from conjugate step functions to conjugate Borel functions.

\begin{lem}\label{lem:monoconj1}
Let $\phi,\psi : \Omega\to \R$ conjugate Borel functions. Then there exists a cover of $\Omega$ by intervals
$$G_1 \leq H_1 \leq F_1 \leq G_2 \leq H_2 \leq \ldots \leq F_N \leq G_{N+1} \leq H_{N+1}$$
sorted in increasing order, such that $N \leq M(\Omega, \delta)$, $\psi$ is nondecreasing on each $G_i$, nonincreasing on each $H_i$, and constant on each $F_i$. Moreover, $d(H_j, G_{j+1}) \geq 2\delta$ for all $j$.
\end{lem}

\begin{proof}
Let $(\phi_n)_n$ a sequence of simple functions converging uniformly to $\phi$. We have already shown that $\phi_n^\uparrow$ converges uniformly  to $\phi^\uparrow = \psi$ and $\phi_n^{\uparrow\downarrow}$ converges uniformly to $\phi^{\uparrow\downarrow}$. Therefore, up to some renaming, we may assume that $(\phi_n,\psi_n)$ is a conjugate pair of step functions such that $(\phi_n, \psi_n) \to (\phi,\psi)$ uniformly. Then for all $n \in \N$, let us take $(F_k^n)_{k=1,\ldots, N_n}$, $(G_k^n)_{k=1,\ldots, N_n+1}$, $(H_k^n)_{k=1,\ldots, N_n+1}$ as in Corollary \ref{cor:monostep}. Up to extraction, $N_n$ being a sequence of integers bounded by $M \doteq M(\Omega,\delta)$, we may assume that it is constant and equal to some $N$. So far we have
$$G_1^n \leq H_1^n \leq F_1^n \leq G_2^n \leq \ldots \leq F_N^n \leq G_{N+1}^n \leq H_{N+1}^n$$
such that $\psi_n$ is nondecreasing on the $G_k^n$'s, decreasing on the $H_k^n$'s and constant on the $F_k^n$'s. For each $k = 1,\ldots, N$, by further extraction, we ensure that the endpoints of all these intervals converge monotonically when $n \to \infty$, which is possible since $\Omega$ is compact. Let us denote by $(G_i)_i$, $(H_i)_i$ and $(F_i)_i$ the limit intervals
\begin{align*}
G_i &= \lim G_i^n,&
H_i &= \lim H_i^n,&
F_i &= \lim F_i^n,&
\end{align*}
in the sense of point set limits (these are well defined by the monotone convergence of their endpoints). If we ignore the endpoints of each limit interval $F_i,\,G_i$ and $H_i$, it is easy to check that these sets give a partition of $\Omega$ and that $\phi$ and $\psi$ will keep the same monotonicity behavior of $\phi_n$ and $\psi_n$ on the interior of each interval.
% (just notice that, if (up to the endpoints of each interval)If $x \in \Omega$, then for $n$ large enough, $x$ is in an interval whose place is fixed, say $H_i^n$, again because of the monotone convergence of the endpoints. Hence $x$ belongs to $H_i$ and $P := \bigcup_i \{G_i, H_i, F_i\}$ is a partition of $\Omega$. Let us check briefly that the monotonicity passes to the limit $\psi$ on each interval. For instance, if $x,y \in H_1$ are such that $x \leq y$, then for $n$ large enough, $x,y \in H_1^n$ and $\psi_n(x) \geq \psi_n(y)$ since $\psi_n$ is decreasing on $H_1^n$. Passing to the limit, one gets $\psi(x) \geq \psi(y)$, hence $\psi$ is decreasing on $H_1$. The same goes for all the intervals of $\mathcal{P}$. Besides, since $d(E_j^n, E_{j+1}^n) \geq 2\delta$, $d(E_j, E_{j+1}) \geq 2\delta$.
\end{proof}

\begin{lem}\label{lem:monoconj2}
 If $\phi,\psi$ are conjugate Borel functions, there exists a finite subdivision of $\Omega$ into intervals such that they are of same monotonicity on each interval of the subdivision.
\end{lem}

\begin{proof}
Take the subdivision
$$G_1 \leq H_1 \leq F_1 \leq G_2 \leq \ldots \leq F_N \leq G_{N+1} \leq H_{N+1}$$
given by Lemma \ref{lem:monoconj1}, where each interval is a monotonicity or constancy interval for $\psi$. Each floor interval $F_i$ may be divided into two parts $F_i^-$ (the first half of the interval) and $F_i^+$ (the second). More precisely, we define $F_i^-=F_i\cap ]-\infty,m_i]$ and $F_i^+=[m_i,+\infty[$, where $m_i:=(\inf G_{i+1}+\sup H_i)/2$ is the middle point between the adjacent endpoints of the two intervals next to $F_i$. In this way both $F_i^-$ and  $F_i^+$ are at least $\delta$ long (in the sense that $m_i-\sup H_i$ and $\inf G_{i+1}-m_i$ are at least $\delta$) and moreover $\psi$ is nondecreasing on each interval of the form $F_i^+\cup G_{i+1}$ (and on the first interval $G_1$) and nondecreasing on each $H_i\cup F_i^-$ (and on $H_{N+1}$).

It is not difficult to check that $\phi$ has the same monotonicity of $\psi$ on these intervals. Let us consider for instance the case of $F_i^+\cup G_{i+1}$, where $\psi$ is nondecreasing. Let us denote by $a$ and $b$ its endpoints, \ie $a=\inf F_i^+$, $b=\sup G_{i+1}$ and $F_i^+\cup G_{i+1}=[a,b]\cap \Omega$. Consider that $\psi$ is also nondecreasing (actually, constant) on $[a-\delta,a]$, since this segment is included in $F_i^-$: this implies that $\phi=\psi^\downarrow$ is nondecreasing as well on the interval $[a,b-\delta]$ (as a consequence of the fact that the behavior of $\phi$ on an interval only depends on the behavior of $\psi$ on the same interval enlarged by $\delta$). We are only left to prove that $\phi$ is also nondecreasing on $]b-\delta,b]$ but this is easy to check since $\phi$ is actually constant on this segment. Indeed, the value $\psi(b)$ is the maximum of $\psi$ on $[b-\delta,b+\delta]$, which implies that $\phi$ is constant on $]b-\delta,b+\delta[$.

An analogous proof works for the intervals $H_i\cup F_i^-$ and for $G_1$ and $H_{N+1}$.
\end{proof}

This last result will be the key point in the proof of the existence of an optimal map, since the following lemma allows for building transport maps which are included in a given strip, provided the boundaries of the strip are given by functions with the same monotonicity.

\begin{lem}\label{lem:monotransp}
Consider two probabilities $\mu\in\pical(\Omega)$ and $\nu\in\pical(\Omega')$ and  two Borel functions $\phi,\psi : \Omega\to \Omega$ with the same monotonicity and such that $\phi \leq \psi$. If there exist $\gamma \in \Pi(\mu,\nu), T \in \Tau(\mu,\nu)$ such that $\Supp \gamma \subset [\phi,\psi]$ and such that $T$ has the same monotonicity as $\phi$ and $\psi$, then the graph of $T$ is $\mu$-almost everywhere included in $[\phi,\psi]$.
\end{lem}

\begin{proof}
Let $a$ be such that for some $x_0$ we have $T(x_0) < a < \phi(x_0)$. One has $ \mu(x : T(x) \leq a)= \nu(y : y \leq a)$ and hence $ \gamma((x,y) : T(x) \leq a) = \gamma((x,y) : y \leq a)$. By subtracting the same quantity $\gamma((x,y): T(x)\leq a,\,y\leq a)$ to these two measures we get
$$\gamma((x,y) : T(x) \leq a < y) = \gamma((x,y) : y \leq a < T(x)).$$
But for $(x,y)$ to be in the right-hand side set, since $T(x_0) < a$, $T$ is nondecreasing and $a < T(x)$, one must have $x > x_0$. Hence $\phi(x) \geq \phi(x_0) > a$ since $\phi$ is also nondecreasing. Therefore
$$\gamma\left((x,y) : y \leq a < T(x)\right) \leq \gamma\left((x,y) : y \leq a < \phi(x)\right)$$
which is null for $\gamma$ is concentrated on $\Gamma \subset [\phi,\psi]$. As a result $\gamma((x,y) : T(x) \leq a < y) = 0$ and $\gamma((x,y) : T(x) \leq a < \phi(x) \leq y)=0$ as well by inclusion. Notice that the condition $\phi(x)\leq y$ is useless here, since $\gamma$ is concentrated on $[\phi,\psi]$, and this also gives
\begin{multline*}
\mu(x : T(x) \leq a < \phi(x))=\gamma((x,y) : T(x) \leq a < \phi(x))\\=\gamma((x,y) : T(x) \leq a < \phi(x) \leq y)=0.
\end{multline*}
This means that for all $a$ such that $T(x_0) < a < \phi(x_0)$ for some $x_0$,
$$\mu(x : T(x) \leq a < \phi(x)) = 0.$$
Taking a countable dense set of such $a$'s, this yields
$$\mu(x : T(x) < \phi(x)) = 0.$$
An analogous proof provides $T\leq \psi$ $\mu-$a.e.
\end{proof}

We are now ready to prove the main theorem of this paper.

\begin{proof}[Proof of Theorem \ref{thm:existmap}]

We replace $f$ and $g$ by $f^{\uparrow\downarrow}$ and $g^\uparrow +K$. Lemma \ref{lem:monoconj2} shows that one can find a finite subdivision of $\Omega$ into intervals $\Omega_1, \ldots, \Omega_N$ on which $f$ and $g$ are of same monotonicity. Let us set $\gamma^\star_j = \gamma^\star_{\vert \Omega_j \times \Omega}$, $\mu_j = \pi_1(\gamma^\star_j)$, $\nu_j = \pi_2(\gamma^\star_j)$ and $f_j = f_{\vert \Omega_j}$, $g_j = g_{\vert \Omega_j}$ so that $\gamma^\star_j \in \Pi(\mu_j,\nu_j)$ is supported in $[f_j,g_j]$. Since the $\mu_j$'s have no atom, it is a classical result that there exists a unique transport map $T_j \in \Tau(\mu_j,\nu_j)$ with the same monotonicity of $f$ and $g$. Since $\Supp \gamma^\star_j \subset [f_j,g_j]$, Lemma \ref{lem:monotransp} guarantees that $T_j$ has its graph $\mu$-almost everywhere included in $[f_j,g_j]$. We shall naturally glue these $T_j$'s together, posing $T(x) = T_j(x)$ on $\Omega_j$. It is clear that $T \in \Tau(\mu,\nu)$ and that $\Supp \gamma_T \subset [f,g]$ which implies by optimality of $[f,g]$ that $\omega_{\delta}(T) = K$. It is an optimal transport map for the cost $\omega_\delta$. 
\end{proof}
\section{A counter-example}

We finish this analysis with an easy counter-example, showing that the optimal transport map $T$ is not always monotone (which would trivialize the interest of the previous existence results). This example is essentially due to J. Louet (\cite{JeanPHD}), who found it for another variational problem. Yet, it can be easily adapted to our scopes.

Consider the map $U:[0,1]\to[0,1]$ given by
$$U(x)=\begin{cases}2x&\mbox{ if }x\leq \frac12,\\
				2-2x&\mbox{ if }x\geq \frac12.\end{cases}$$
Consider $\mu=f(x)dx$ a probability measure on $[0,1]$ given by the density $f$:
$$f(x)=\begin{cases}\frac 85&\mbox{ if }x\in\left[0,\frac 14\right]\cup\left[\frac 34,1\right],\\
				\frac 25&\mbox{ if }x\in\left[\frac 14,\frac 34\right].\end{cases}$$
Take $\nu=U_\#\mu$. It is not difficult to check that $\nu$ is supported on $[0,1]$, and is absolutely continuous with density $g$				
$$g(x)=\begin{cases}\frac 85&\mbox{ if }x\in\left[0,\frac 12\right],\\
				\frac 25&\mbox{ if }x\in\left[\frac 12,1\right].\end{cases}$$				
Also, one can compute the unique monotone increasing map $T$ such that $T_\#\mu=\nu$. Its expression is
$$T(x)=\begin{cases}x&\mbox{ if }x\in\left[0,\frac 14\right],\\
				\frac 14 + \frac 14(x-\frac14)&\mbox{ if }x\in\left[\frac 14,\frac 34\right],\\
				 x-\frac 38&\mbox{ if }x\in\left[\frac 34,\frac 78\right],\\
				\frac 12 + 4(x-\frac 78)&\mbox{ if }x\in\left[\frac 78,1\right].\end{cases}$$
The unique monotone decreasing map is simply symmetric to $T$, due to the symmetry of the starting measure $\mu$.

Consider now $\delta\leq \frac 18$. It is easy to check that we have $\omega_\delta(T)=4\delta$ (a slope of $4$ is realized in the last interval $\left[\frac 78,1\right]$, whose length is $\frac 18$), while 
$\omega_\delta(U)=2\delta$ (since $U$ has always slope $2$).

This proves that, for these given choices of $\mu,\nu$ and $\delta$, the optimal map cannot be $T$ (it does not prove on the contrary that the optimal map is $U$).
\bigskip

\begin{minipage}{5.5cm}
\begin{tikzpicture}[xmin=-1,xmax=4.5,ymin=-1,ymax=4.5]
 \axes \fenetre

\draw[red, domaine={0}{2}, samples=100] plot (\x,{2*\x});
\draw[red, domaine={2}{4}, samples=100] plot (\x,{8-2*\x});
\draw (1,0) node{$\bullet$} node[above]{$\frac 14$};
\draw (3,0) node{$\bullet$} node[above]{$\frac 34$};;
\draw (0,2) node{$\bullet$} node[right]{$\frac 12$};;
\draw (0.5,-0.5) node{$f=\frac 85$};
\draw (3.5,-0.5) node{$f=\frac 85$};
\draw (2,-0.5) node{$f=\frac 25$};
\draw (-0.5,1) node{$g=\frac 85$};
\draw (-0.5,3) node{$g=\frac 25$};
\draw[dotted] (1,0)--(1,4);
\draw[dotted] (3,0)--(3,4);
\draw[dotted] (0,2)--(4,2);
\draw (3.5,3.5) node{$U(x)$};
\end{tikzpicture}
\end{minipage}
\hspace{1cm}
\begin{minipage}{5.5cm}
\begin{tikzpicture}[xmin=-1,xmax=4.5,ymin=-1,ymax=4.5]
 \axes \fenetre

\draw[red, domaine={0}{1}, samples=100] plot (\x,{\x});
\draw[red, domaine={1}{3}, samples=100] plot (\x,{\x/4+3/4});
\draw[red, domaine={3}{3.5}, samples=100] plot (\x,{\x-3/2});
\draw[red, domaine={3.5}{4}, samples=100] plot (\x,{4*\x-12});

%\draw[red, domaine={2}{4}, samples=100] plot (\x,{8-2*\x});
\draw (1,0) node{$\bullet$} node[above]{$\frac 14$};
\draw (3,0) node{$\bullet$} node[above]{$\frac 34$};;
\draw (0,2) node{$\bullet$} node[right]{$\frac 12$};;
\draw (0.5,-0.5) node{$f=\frac 85$};
\draw (3.5,-0.5) node{$f=\frac 85$};
\draw (2,-0.5) node{$f=\frac 25$};
\draw (-0.5,1) node{$g=\frac 85$};
\draw (-0.5,3) node{$g=\frac 25$};
\draw[dotted] (1,0)--(1,4);
\draw[dotted] (3,0)--(3,4);
\draw[dotted] (0,2)--(4,2);
\draw (2.5,3.5) node{$T(x)$};
\end{tikzpicture}
\end{minipage}

\section{Technical extensions}

For the sake of simplicity, we tried to describe our problem sticking to the easiest case. For instance, we assumed both $\Supp(\mu)$ and $\Supp(\nu)$ to be compact, which simplified some proofs.

Also, we used the same set $\Omega$ both as the support of $\mu$ and as the domain where the functions are defined and the oscillation $\omega_\delta$ is computed. Indeed, if we take a measure which is not fully supported on a set $\Omega$, we could face two reasonable choices for the functional $\omega_\delta$, since one could take the supremum over pairs of point $x,x'\in \Supp(\mu)$ with $|x-x'|<\delta$, or more generally over $x,x'\in \Omega$ with $|x-x'|<\delta$. We chose the first definition, which is easier to handle and corresponds more to the application we had in mind.
 
Yet, it is true that the behavior of $T$ on $\Omega\setminus\Supp(\mu)$ does not affect the image measure constraint, but it could affect the value of $\omega_\delta$ if the second definition is chosen, thus penalizing big jumps of $T$ between different connected components of $\Supp(\mu)$. The problem is that this second definition makes it more difficult to define a Kantorovich approach and we need to slightly change our functional $\omega^K_\delta$.

\subsection{Optimal plans and maps in the non-compact case}

Few adaptations have to be performed for the existence of an optimal plan if $\mu,\nu$ are not compactly supported. 

First, let us notice that in this case the existence of a plan $\gamma\in\Pi(\mu,\nu)$ such that $\omega^K_\delta(\gamma)<+\infty$ is not straightforward (unless $\Supp(\nu)$ is compact) and has to be supposed. This fact more or less corresponds to the fact that $\nu$ has a queue which is comparable to (or smaller) than that of $\mu$, \ie that there exists a constant $k$ such that $\nu(B_{kr}^c)\leq \mu(B_r^c)$ for large $r$. Anyway, let us assume that $\inf (K_\delta)<+\infty$.

It is standard and well-known in optimal transport (see for instance \cite{villani}) that for given $\mu,\nu$ the set $\Pi(\mu,\nu)$ is a tight subset of $\pical(\Omega\times\Omega)$. This still allows for the extraction of a weakly converging subsequence $\gamma_n\deb\gamma$.

Obviously we cannot take a subsequence such that $\Gamma_n$ Hausdorff converges to $\Gamma$ since this would require the domain to be compact. Yet, we can easily, by a diagonal argument, extract a subsequence such that for every natural integer $R$ we have $\Gamma_n\cap(\overline{B_R}\times\overline{B_R})\to \Gamma_R\subset \overline{B_R}\times\overline{B_R}$.

Since we can write $\omega^K_\delta=\sup_R\omega^K_{\delta,R}$, where 
$$\omega^K_{\delta,R}(\gamma):=
 \sup \left\{\abs{y-y'} : (x,y), (x',y') \in \Supp(\gamma)\cap\left(\overline{B_R}\times\overline{B_R}\right), \abs{x-x'}<\delta\right\},$$
we can infer the semicontinuity of $\omega^K_\delta$ from that of each $\omega^K_{\delta,R}$, which can be proven using $\Supp(\gamma) \cap(\overline{B_R}\times\overline{B_R})\subset \Gamma_R$ and applying the same arguments as above.

Once the existence of an optimal plan is established, one needs to adapt the content of Section 3 to the case where $\Omega$ is non-compact. This is not difficult once we notice that the support $\Gamma$ of an optimal plan $\gamma$ must be ``locally bounded'' in the following sense.

If we suppose that the minimum of $\omega^K_\delta$ is finite, \ie $K<+\infty$, then every vertical fiber $ \set{y : (x,y) \in \Gamma}$ has diameter bounded above by $K$ (since if $(x,y)$ and $(x,y')$ belong to $\Gamma$, then we should have $|y-y'|\leq K$ due to $|x-x|=0<\delta$). This proves that $f$ and $g$ are well defined. Not only, for every $x$ $f$ is bounded above by $f(x)+K$ on the whole ball $B(x,\delta)$, which allows, by recursively applying this bound, to say that $f$ is locally bounded. Analogous considerations hold for $g$. In particular, the intersection of the support of $\Gamma$ with vertical strips of the form $\overline{B_R}\times\R$ are bounded and hence compact.

The only extra point to remark in order to perform the same analysis on $\R$ is that we will have no more a finite number of intervals: when we approximate (Proposition \ref{prop:transfregul}) $\phi$ and $\psi$ with step functions this will be done with functions which are constant on a countable (but locally finite) number of intervals; the sum in Lemma \ref{lem:step} will be no longer finite but locally finite, and the bounds on the number of intervals appearing in Corollary \ref{cor:monostep} and on will only be local. Yet, the main points of the proof will stay the same, since they are essentially local.

\subsection{Wider definition of $\omega_\delta$}

If one considers a measure $\mu$ which is not fully supported on $\Omega$ but wants to define $\omega_\delta$ in the following way 
$$\omega_\delta(T):= \sup_{x,x'\in\Omega,\,|x-x'|<\delta}|T(x)-T(x')|,$$
then the definition of the functional $\omega^K_\delta$ has to be changed, since that of Section 2.1 only considers pairs $(x,y)$ and $(x',y')$ in the support of $\gamma$, so that $x,x'\in\Supp(\mu)$. 

A possible way to overcome the problem is the following: define
$$\omega_\delta^K(\gamma):=\inf\left\{ \sup\{|y-y'|\;:\;(x,y),(x',y')\in \Gamma,\, |x-x'|<\delta\}\;:\;\Gamma\in\mathcal{A}(\gamma)\right\},$$
where
$$\mathcal{A}(\gamma):=\{\Gamma\subset\Omega\times\R\;:\;
\Gamma\supset\Supp(\gamma),\,\pi_1(\Gamma)=\Omega\}.$$

This means that, instead of computing a maximal oscillation on the support of $\gamma$, we compute it on sets which are extensions of this support, but have full projection onto $\Omega$, and we chose the best possible extension. Again, this functional only depends on $\Supp(\gamma)$, as in Section 2.1.

The existence of an optimal $\gamma$ both in the compact or non-compact case easily follows from the same considerations, simply replacing the support with this set $\Gamma$ (in a minimizing sequence, take a sequence of sets $\Gamma_n$, make it converge to a set $\Gamma$, which will contain the support of the limit measure\dots).

Moreover, for any fixed measure $\gamma$ it is straightforward that an optimal set $\Gamma$ does exist (same argument), and we will use this set to define the functions $f$ and $g$ of Section 3.1. The rest of the construction is exactly the same.

\bigskip
\bigskip

{\small

\begin{minipage}{6cm}
Didier Lesesvre, Paul Pegon\\
\'Ecole Normale Supérieure de Cachan,\\
 61, Avenue du Président Wilson,\\
 94235 Cachan cedex, FRANCE,\\
  {\tt didier.lesesvre@ens-cachan.fr\\paul.pegon@ens-cachan.fr}
  \end{minipage}
\begin{minipage}{7cm}
Filippo Santambrogio,\\
Laboratoire de Mathématiques d'Orsay,\\
 Université Paris-Sud,\\
 91405 Orsay cedex, FRANCE,\\
  {\tt filippo.santambrogio@math.u-psud.fr}\\
  \phantom{a}
  \end{minipage}
  
  }
\end{document}